\documentclass[reqno]{amsart}
\usepackage{txfonts}
\usepackage{amsfonts}
\usepackage{amssymb}
\usepackage[mathscr]{eucal}
\usepackage{amsmath}
\usepackage[bookmarksopen,colorlinks,citecolor=blue,
linkcolor=black,pdfstartview=FitH]{hyperref}
\usepackage{ytableau}
\usepackage{float}
\usepackage{url}
\usepackage{appendix}
\usepackage[all]{xy}

\newtheorem{theorem}{Theorem}[section]
\newtheorem{corollary}[theorem]{Corollary}
\newtheorem{lemma}[theorem]{Lemma}
\newtheorem{proposition}[theorem]{Proposition}
\theoremstyle{definition}
\newtheorem{definition}[theorem]{Definition}

\newtheorem{example}[theorem]{Example}

\newtheorem{question}[theorem]{Question}
\newtheorem{conjecture}{Conjecture}
\numberwithin{equation}{section}
\theoremstyle{remark}
\newtheorem{remark}[theorem]{Remark}

\allowdisplaybreaks[4]

\begin{document}

\title{On the local dimensions of solutions of Brent equations}
\pdfbookmark[0]{On the local dimensions of Brent equations}{}

\author{Xin Li}
\address{Department of Mathematics, Zhejiang University of Technology, Hangzhou 310023, P. R. China}
\email{xinli1019@126.com (X. Li)}

\author{Yixin Bao}
\address{School of Sciences, Harbin Institute of Technology, Shenzhen, 518055, China}
\email{mabaoyixin1984@163.com (Y. Bao)}

\author[zhang]{Liping Zhang}
\address{School of Mathematical Sciences, Qufu Normal University, Qufu 273165, P. R. China}
\email{zhanglp06@gmail.com (L. Zhang)}


\subjclass[2010]{Primary 14L30; Secondary 68Q17}
\keywords{Brent equations, Local dimension,  Jacobian matrix, Automorphism group, Matrix multiplication tensor}

\begin{abstract}
Let $\langle m,n,p \rangle$ be the matrix multiplication tensor.
The solution set of Brent equations corresponds to the tensor decompositions of $\langle m,n,p \rangle$. We study the local dimensions of solutions of the Brent equations over the field of complex numbers. The rank of Jacobian matrix of Brent equations provides an upper bound of the local dimension, which is well-known. We calculate the ranks for  some typical known solutions, which are provided in the databases \cite{Faw22+} and \cite{Heule19}. We show that the automorphism group of
the natural algorithm computing $\langle m,n,p \rangle$ is
$(\mathcal{P}_m\times \mathcal{P}_n\times \mathcal{P}_p)\rtimes Q(m,n,p)$, where $\mathcal{P}_m$, $\mathcal{P}_n$ and $\mathcal{P}_p$ are  groups of generalised permutation matrices, $Q(m,n,p)$ is a subgroup of $S_3$ depending on $m$, $n$ and $p$.
For other algorithms computing $\langle m,n,p \rangle$, some conditions  are given, which imply the corresponding automorphism groups are isomorphic to subgroups of $(\mathcal{P}_m\times \mathcal{P}_n\times \mathcal{P}_p)\rtimes Q(m,n,p)$. So under these conditions,
$m^2+n^2+p^2-m-n-p-3$ is a lower bound for the local dimensions of solutions of Brent equations. Moreover, the gap between the lower and upper bounds is discussed.
\end{abstract}
\maketitle

\section{Introduction}\label{sec:intro}


The solution set $V$ of a system of polynomial equations may contain  points, curves, surfaces, etc., each with its own dimension. A fundamental problem when working with such systems is to describe these (local) dimensions \cite{Bates09}. Let $V$ be an affine variety and $q\in V$.  The \emph{local dimension} of $q$ is defined by the maximum dimension of an irreducible component of $V$ containing $q$ \cite{cox15}.

In this paper, we give estimations for the local dimensions of the solution set of \emph{the Brent equations} \cite{Brent70,Heule21,Lands-17,Smir-13}. Since the Brent equations contain a huge number of equations and variables, it is hard to calculated the local dimensions explicitly by current methods, such as \cite{Bates09,Lai-21}. So in this paper, we study the upper and lower bounds of the local dimensions. Firstly, by the Jacobian Criterion of \cite[Theorem 4.1.12]{DS07}, we calculate the upper bound of local dimensions for some typical points, which are provided in the databases \cite{Faw22+} and \cite{Heule19}. Secondly, by the theory of linear algebraic groups, we get the lower bound of local dimensions \cite[Proposition 1.11]{Brion10}. Other methods for the estimation of local dimensions can be found in \cite{Breid18}.


Let $\langle m,n,p \rangle$ denote the  well-known matrix multiplication tensor (see e.g. \cite{Bur15,deGr87,Heule21,Lands-17}). The tensor decomposition of $\langle m,n,p \rangle$ of length $r$ gives rise to the Brent equations \cite{Heule21,Smir-13} (see also Section \ref{subsec:brent}). We  denote it by $B(m,n,p;r)$. The solution set of $B(m,n,p;r)$ in the field of complex number is denoted by $V(m,n,p;r)$. For a point $q\in V(m,n,p;r)$, let $\dim_q V(m,n,p;r)$ denote its local dimension. Let $J_q(B(m,n,p;r))$ denote the Jacobian matrix of $B(m,n,p;r)$ at $q$ and $rank~J_q(B(m,n,p;r))$ its rank. The Jacobian Criterion implies $$\dim_q V(m,n,p;r)\leq k-rank~J_q(B(m,n,p;r))$$ where $k=(mn+np+pm)r$. In Section \ref{ssubsec:rk3344}, we calculated
$rank~J_q(B(3,3,3;23))$ and $rank~J_{q'}(B(4,4,4;49))$, where $q$ and $q'$ are the solutions summarized in the databases \cite{Faw22+} and \cite{Heule19}. So we can get upper bounds of local dimensions of these points.  For the solutions provided in the database \cite{Heule19}, the distribution of $rank~J_q(B(3,3,3;23))$ is summarized in Table \ref{tab:333-1} and \ref{tab:333-2} of Section \ref{ssubsec:rk3344}. We can see that the upper bounds given by the Jacobian Criterion for Brent equations are meaningful, because the ranks of $J_q(B(3,3,3;23))$ and $J_{q'}(B(4,4,4;49))$ are big.

On the other hand, the isotropy group of $\langle m,n,p \rangle$ induces a group action on $V(m,n,p;r)$ \cite{deGr78-1,deGr78-2, deGr87}. By the theory of linear algebraic groups \cite[Proposition 1.11]{Brion10}, we get a lower bound for local dimensions of points in $V(m,n,p;r)$. Let $\mathcal{A}(m,n,p)=\{e_{ij}\otimes e_{jk}\otimes e_{ki}|1\leq i\leq m, 1\leq j\leq n, 1\leq k\leq p\}$
denote the \emph{natural algorithm} computing $\langle m,n,p \rangle$. We show that the automorphism group of $\mathcal{A}(m,n,p)$ is $(\mathcal{P}_m\times \mathcal{P}_n\times \mathcal{P}_p)\rtimes Q(m,n,p)$, where  $\mathcal{P}_m\subseteq GL_m(\mathbb{C})$ is the generalised permutation groups (similarly, for $\mathcal{P}_n$ and $\mathcal{P}_p$), $Q(m,n,p)$ is a subgroup of $S_3$ depending on $m,n,p$ (see Corollary \ref{cor:autnat}).
Some conditions (see e.g. Definition \ref{def:d-p} and \ref{def-wdp}) are given which imply that the automorphism group of an  algorithm computing $\langle m,n,p \rangle$ is (isomorphic to) a subgroup of $(\mathcal{P}_m\times \mathcal{P}_n\times \mathcal{P}_p)\rtimes Q(m,n,p)$ (see Corollary \ref{cor:autmnpq} and Proposition \ref{prop-wdaut}).
So under these conditions, we obtain a lower bound for local dimensions for points $q$ in $V(m,n,p;r)$ (see Corollary \ref{cor:lbbe} and Proposition \ref{prop-wdaut}):
$$\dim_q V(m,n,p;r)\geq m^2+n^2+p^2-m-n-p-3.$$
If we take account into another two (trivial) linear group actions which all tensors have, then the lower bound above can be improved to $$\dim_q V(m,n,p;r)\geq m^2+n^2+q^2+2r-m-n-p-3,$$
or even
$$\dim_q V(m,n,p;r)\geq m^2+n^2+q^2+2r-3,$$
which can be seen at equation (\ref{eq-lqq}) and (\ref{eq-lqqq}). The gap between the lower and upper bounds and some numerical experimental results are discussed.


Moreover, we find that the polynomial ideals generated by Brent equations may not be radical. The ranks of Jacobian matrices may serve as invariants to distinguish group orbits.

The paper is organized as follows. In Section \ref{sec:general}, the preliminaries about local dimension are recalled, which are the Jacobian Criterion and the property for the dimension  of group orbits. The main results are given in Section \ref{sec:brent}.

%

\section{Upper and lower bounds of the local dimensions}\label{sec:general}
\subsection{Rank of Jacobian matrix and the upper  bound}
Let  $\mathbb{C}$ denote the field of complex numbers and $\mathbb{C}^n$ the $n$-dimensional complex vector space. Let $A=\mathbb{C}[x_1,...,x_n]$ be the polynomial ring in $n$ variables over $\mathbb{C}$. If $T$ is any subset of $A$, set
$$V(T)=\{p\in \mathbb{C}^n|f(p)=0~for~all~f\in T\}$$
to be the common zeros of all the elements of $T$.

\begin{definition}
A subset $Y$ of $\mathbb{C}^n$ is an affine variety if there exists a subset $T\subseteq A$ such that $Y=V(T)$.
\end{definition}
\begin{remark}
In literatures, $Y$ is also called an algebraic set (see e.g. \cite{DS07}). Here we follow the definition of \cite{cox15}.
\end{remark}

Thus, the solution set $V$ of a system of polynomial equations is an affine variety. Given a  point $p$ on the affine variety $V =\{x\in \mathbb{C}^n |F(x)=0\}$ for a set of polynomials $F(x)=\{f_1(x),..., f_m(x)\}\subseteq \mathbb{C}[x_1,...,x_n]$. The Jacobian matrix of $F$ at a point $p\in \mathbb{C}^n$ is defined by
$$J_p(F)=J_p(f_1,f_2,...,f_m)=\left(
           \begin{array}{ccc}
\frac{\partial f_1}{\partial x_1}(p) & \cdots & \frac{\partial f_1}{\partial x_n}(p) \\
 \vdots & ~ & \vdots \\
\frac{\partial f_m}{\partial x_1}(p) & \cdots & \frac{\partial f_m}{\partial x_n}(p) \\
           \end{array}
         \right)_{m\times n}$$
The \emph{vanishing ideal} of $V$  is given by the set
$$I(V)=\{f\in \mathbb{C}[x_1,...,x_n]|~f(a)=0~ \text{for all}~a\in V\},$$
which is a radical ideal. Let $\langle f_1,f_2,...,f_m\rangle$ denote the ideal generated by $f_1,f_2,...,f_m$ in $\mathbb{C}[x_1,...,x_n]$.
Then $I(V)$ is the radical of $\langle f_1,f_2,...,f_m\rangle$.
So $\langle f_1,f_2,...,f_m\rangle$ may be strictly contained in $I(V)$. Let $T_p(V)$ denote the tangent space of $p$ at $V$ (see Definition 4.1.6 of \cite{DS07}). If $I(V)$ is generated by
$\{g_1,g_2,...,g_r\}$, then by Remark 4.1.7 of \cite{DS07} we know that
$$\dim T_p(V)=n-rank~J_p(g_1,g_2,...,g_r).$$

For an affine variety $V$ which is defined by $F=\{f_1,f_2,...,f_m\}$, we have the following upper bound on the local dimension of a point $p\in V$, which is called \emph{Jacobian Criterion} in \cite{DS07}.
\begin{theorem}\cite[Theorem 4.1.12, Jacobian Criterion]{DS07}\label{th:updim}
Let $p$ be a point on an affine variety $V \in \mathbb{C}^n$. Then $$\dim_p V\leq \dim T_p(V)\leq n-rank~J_p(F).$$
\end{theorem}
In Section \ref{sec:brent}, when $F$ is the Brent equations, using computer we will calculate $J_p(F)$ for some typical known solutions.
So they give the upper bounds of local dimensions for these solutions.

\subsection{Isotropy group orbit and the lower bound}\label{subsec:isogolb}
In this section, we recall some results of \cite{Brion10} and \cite{deGr87}.

A $G$-variety is an affine variety $X$ equipped with an action of the algebraic group $G$. Given a $G$-variety $X$ and a point $x \in X$, the orbit $G\cdot x\subseteq X$ is the set of all $g \cdot x$, where $g \in G$. The automorphism group (it was called the isotropy group in \cite{Brion10}) $G_x \subseteq G$ is the set of those $g \in G$ such that $g \cdot x = x$; it is a closed subgroup of $G$.

\begin{proposition}\cite[Proposition 1.11]{Brion10}\label{prop:ldimgo}
The group orbit $G\cdot x$ is a locally closed, smooth subvariety of $X$, and every component of $G\cdot x$ has dimension $\dim(G)-\dim (G_x)$.
\end{proposition}

A group $G$ is said to be a \emph{semidirect product} of $A$ by $B$, which is denoted by $G=A\rtimes B$, if $G=AB$, $A\unlhd G$ ($A$ is a normal subgroup of $G$) and $A\cap B=1$. If $A$ is a  subgroup (or isomorphic to a subgroup) of $G$, then it will be denoted by $A\leq G$.  Let $PGL_n(\mathbb{C})$ and $\mathcal{D}_n\subseteq GL_n(\mathbb{C})$ denote the projective general linear group and the group of nonsigular diagonal matrices over $\mathbb{C}$, respectively. Let $S_n$ denote the symmetric group of order $n$. For the dimension of algebraic groups, we have the following proposition. They can be found in Section 1.8 and 5.5 of \cite{Spr-98} and \cite{Mur-05}.
\begin{proposition}\label{prop:dimag}
Suppose that $G$, $H$ are algebraic groups. Then
\begin{enumerate}
  \item  $\dim G=0$ if $G$ is a finite group;
  \item $\dim (G\times H)=\dim (G\rtimes H)=\dim G+\dim H$;
  \item In particular, $\dim PGL_n(\mathbb{C})=n^2-1$;
  \item\label{prop:dimag-item-mnp}  $\dim ((PGL_m(\mathbb{C})\times  PGL_n(\mathbb{C})\times PGL_p(\mathbb{C}))\rtimes S_3)=m^2+n^2+p^2-3$;
  \item $\dim ( \mathcal{D}_n\rtimes S_n) =n$.
\end{enumerate}
\end{proposition}

Let $U$, $V$, $W$ be finite dimensional vector spaces over $\mathbb{C}$ and  $\Phi:~U \times V \to W$ a bilinear mapping. Let $t\in U^*\otimes V^* \otimes W$ be the structure tensor of $\Phi$ . According to some choice of bases in $U$, $V$, $W$, let $t_{\lambda\mu\nu}$ be the coordinates of $t$, where $\lambda=1,...,\dim U$; $\mu=1,...,\dim V$; $\nu=1,...,\dim W$.
Suppose that $t$ has a tensor (or decomposable) decomposition of length $R$. That is, $t=\sum_{i=1}^R u_i\otimes v_i\otimes w_i$ where $u_i=(u_{i\lambda})\in U^*\cong U$, $v_i=(v_{i\mu})\in V^*\cong V$, $w_i=(w_{i\nu})\in W$. It means the following polynomial system has solutions:
\begin{equation}\label{eq:algpoly}
t_{\lambda\mu\nu}-\sum_{i=1}^R u_{i\lambda}v_{i\mu}w_{i\nu}=0.
\end{equation}
Let $\mathfrak{a}_R(\Phi)$ denote the set of all solutions which is an affine variety. It was called the \emph{algorithm variety} of $\Phi$ \cite{deGr87}. In particular, if $\Phi$ is the matrix multiplication, then the polynomial system (\ref{eq:algpoly}) is called \emph{the Brent equations} (see an explicit description in equation (\ref{eq:brent})).

If some linear transformation $\varphi\in End(U^*\otimes V^* \otimes W)$ leaves the structure tensor $t$ of $\Phi$ fixed, i.e. $\varphi(t)=t$, then it induces a transformation (denoted by $\hat{\varphi}$): $\hat{\varphi}:\mathfrak{a}_R(\Phi)\to \mathfrak{a}_R(\Phi)$. In particular, let
$$H=\{\varphi|~\varphi(t)=t~\text{and}~\varphi\in GL(U^*\otimes V^* \otimes W)\}.$$
Then $H\leq GL(U^*\otimes V^* \otimes W)$ is called \emph{the isotropy group} of $t$. The transformation $\hat{\varphi}$ induced by $\varphi\in H$ is a group action on the affine variety $\mathfrak{a}_R(\Phi)$. For a point $x\in \mathfrak{a}_R(\Phi)$, if we know the dimension of the automorphism group $H_x$ and $H$, then with these notations, by Proposition \ref{prop:ldimgo} we have the following proposition
\begin{proposition}\label{prop:lowerb}
For $x\in \mathfrak{a}_R(\Phi)$, the local dimension of $x$ is greater than $\dim H-\dim H_x$.
\end{proposition}


In particular, if $\Phi$ is the matrix multiplication and $H$ the isotropy group studied in \cite{deGr78-1,deGr78-2} and \cite{Bur15}, then in Section \ref{subse:auto} we will study the structure of $H_x$. A lower bound of local dimensions of points in $\mathfrak{a}_R(\Phi)$ (the solution of Brent equations) will be given.

\subsection{Rank of the Jacobian matrix as an invariant}

Suppose that $V$ is an affine variety and $G$ is an algebraic group which acts on $V$. If we know the generators of the radical ideal $I(V)$,  then by the following proposition we can see that the rank of Jacobian matrix of the generators is an invariant to distinguish group orbits.

\begin{proposition}\cite[Cor. 4.1.18]{DS07}\label{prop:isotp}
If $\varphi: A\to B$ is an isomorphism of affine varieties
and $p \in A$ is a point, then
$$d_p \varphi: T_p A \to T_{\varphi(p)}B$$
is an isomorphism of $\mathbb{C}$-vector spaces.
\end{proposition}

\begin{corollary}\label{cor:rkorb}
Suppose that $V\subseteq \mathbb{C}^n$ is an affine variety and $G\subseteq GL_n(\mathbb{C})$ is an algebraic group which acts on $V$. Suppose that $I(V)$ is generated by $\{f_1,f_2,...,f_m\}\subseteq \mathbb{C}[x_1,x_2,...,x_n]$. Let
$p,q\in V$ be two points. If $rank~J_p(f_1,f_2,...,f_m)\neq rank~ J_q(f_1,f_2,...,f_m)$, then $p$ and $q$ can not lie in the same group orbit of $G$.
\end{corollary}
\begin{proof}
In Proposition \ref{prop:isotp}, if we let $A=B=V$, then the group action of $G$ induces an isomorphism of $V$. So if $p$ and $q$  lie in the same group orbit of $G$, we have $\dim T_p V =\dim T_{q}V$. The proof is completed by  $\dim T_p V=n-rank~J_p(f_1,f_2,...,f_m)$ for all $p\in V$.
\end{proof}

\section{The local dimensions of solutions of Brent equations}\label{sec:brent}
In this section, we use the definitions and symbols of \cite{Bur15}.
Let $V_1$, $V_2$, $V_3$ be vector spaces over $\mathbb{C}$. Denote their tensor product by $\tilde{V}=V_1\otimes V_2\otimes V_3$. A tensor $t\in\tilde{V}$ is  \emph{decomposable} if $t=v_1\otimes v_2\otimes v_3$ where $v_i\in V_i$. If a tensor $t\in\tilde{V}$ can be written as $$t=t_1+t_2+\cdots +t_r$$
where  $t_i$ are decomposable tensors, we call $t$ has a \emph{tensor decomposition of length $r$.} The set $\{t_1,...,t_r\}$ is called \emph{an algorithm of length $r$ computing $t$}. The minimal length computing $t$ is called the \emph{tensor rank} of $t$.

\subsection{Brent equations and the upper bound}\label{subsec:brent}
In particular,
let $V_1=\mathbb{C}^{m\times n}$, $V_2=\mathbb{C}^{n\times p}$ and $V_3=\mathbb{C}^{p\times m}$ denote the set of complex $m\times n$, $n\times p$  and $p\times m$ matrices, respectively.
The matrix multiplication tensor is defined by (see e.g. \cite{Bur15,Lands-17})
\begin{equation}\label{eq:defmmt}
\langle m,n,p \rangle=\sum_{i=1}^m\sum_{j=1}^n\sum_{k=1}^p e_{ij}\otimes e_{jk}\otimes e_{ki}\in \mathbb{C}^{m\times n}\otimes \mathbb{C}^{n\times p}\otimes \mathbb{C}^{p\times m},
\end{equation}
where $e_{ij}$ denotes the matrix with a 1 in the $i$th row and $j$th column, and other entries are 0. So equation (\ref{eq:defmmt}) provides an algorithm of length $mnp$ computing $\langle m,n,p \rangle$. Starting from the famous work of Strassen \cite{Stra-69}, finding possible tensor decomposition of $\langle m,n,p \rangle$ with length less than $mnp$ is an important topic in algebraic complexity theory \cite{BCS97,deGr87,Lands-17}. It is equivalent to find solutions of a system of polynomial equations called the \emph{Brent equations}, which was firstly observed by Brent \cite{Brent70}.

The Brent equations are given as follows (see also e.g. \cite{Bur15,Heule21,Smir-13}). To find a tensor decomposition of $\langle m,n,p \rangle$ of length $r$, it means that
\begin{equation}\label{eq:mnpr}
\langle m,n,p \rangle=t_1+t_2+\cdots+t_r,
\end{equation}
where $t_i=u_i\otimes v_i\otimes w_i$ and  $$u_i=\left(\alpha_{i_1,i_2}^{(i)}\right)\in\mathbb{C}^{m\times n}, v_i=\left(\beta_{j_1,j_2}^{(i)}\right)\in\mathbb{C}^{n\times p}, w_i=\left(\gamma_{k_1,k_2}^{(i)}\right)\in\mathbb{C}^{p\times m}.$$
Here $\alpha_{i_1,i_2}^{(i)}$, $\beta_{j_1,j_2}^{(i)}$ and $\gamma_{k_1,k_2}^{(i)}$ are unknown variables.
Then by (\ref{eq:defmmt}) and (\ref{eq:mnpr}) we obtain the Brent equations which correspond to the tensor decomposition of $\langle m,n,p \rangle$ of length $r$:
\begin{equation}\label{eq:brent}
\sum_{i=1}^{r} \alpha_{i_1,i_2}^{(i)}\beta_{j_1,j_2}^{(i)}\gamma_{k_1,k_2}^{(i)}=
\delta_{i_2,j_1}\delta_{j_2,k_1}\delta_{k_2,i_1},
\end{equation}
where $k_2,i_1\in \{1,2,...,m\}$, $i_2,j_1\in\{1,2,...,n\}$,
$j_2,k_1\in\{1,2,...,p\}$ and $\delta_{ij}$ is the Kronecker symbol.

In the following, we denote the Brent equations in (\ref{eq:brent}) by $B(m,n,p;r)$. From (\ref{eq:brent}), we can see that the polynomial system $B(m,n,p;r)$ has $(mnp)^2$ equations and $(mn+np+pm)r$ variables. So it is a large polynomial system even when $m,n,p$ are small. Usually, $B(m,n,p;r)$ is an overdetermined polynomial system. For example, $B(3,3,3;23)$ (resp. $B(4,4,4;49)$) has 729 (resp. 4096) equations and 621 (resp. 2352) variables. The solution set of $B(m,n,p;r)$ over $\mathbb{C}$ is denoted by $V(m,n,p;r)$.


\subsubsection{The numerical statistics of ranks of $J_p(B(3,3,3;23))$ and $J_p(B(4,4,4;49))$}\label{ssubsec:rk3344}
\

In this subsection, we calculate the ranks of Jacobian matrices of points in  $V(3,3,3;23)$ and $V(4,4,4;49)$, which are provided in \cite{Heule19} and \cite{Faw22+}. Then by Theorem \ref{th:updim}, we get upper bounds of local dimensions for these points. We mainly focused on points of $V(3,3,3;23)$ provided in \cite{Heule19}.

In \cite{Heule19}, Heule et al. provided 17376 points in $V(3,3,3;23)$ (they are called ``schemes'' in \cite{Heule19}). Using computer, we calculated the ranks of Jacobian matrices of $B(3,3,3;23)$ at these points. We summarize the results in Table \ref{tab:333-1} and \ref{tab:333-2}.  For points $p\in V(3,3,3;23)$ provided in \cite{Heule19}, the first rows of Table \ref{tab:333-1} and \ref{tab:333-2} are the ranks of $J_p(B(3,3,3;23))$. The second rows are the upper bound of local dimensions provided in Theorem \ref{th:updim}, that is, $621-J_p(B(3,3,3;23))$. The third rows are total number of points that have the rank among all 17376 points. Take the first column of Table \ref{tab:333-1} as an example. The first number 526 is the rank of Jacobian matrix at some points in $V(3,3,3;23)$, which are provided in \cite{Heule19}. The second number is 621-526=95, which is an upper bound of local dimensions for these points. The third number 5 is the number of points that have the rank 526 among all 17376 points.


\begin{table}[h]
\begin{tabular}{|l|l|l|l|l|l|l|l|l|l|l|}
\hline
Rank & 526 & 527 & 528 & 529 & 530 & 531 & 532 & 533 & 534 & 535 \\ \hline
Upper bound & 95 & 94 & 93 & 92 & 91 & 90 & 89 & 88 & 87 & 86 \\ \hline
Total number  & 5 & 25 & 79 & 256 & 624 & 1421 & 2250 & 3069 & 3486 & 2870  \\ \hline
\end{tabular}
\caption{Part \uppercase\expandafter{\romannumeral1}}\label{tab:333-1}
\vspace{-1.5em}
\end{table}

\begin{table}[h]
\begin{tabular}{|l|l|l|l|l|l|l|l|l|l|l|}
\hline
Rank & 536 & 537 & 538 & 539 & 540 & 541 & 542 & 543 & 544 &545  \\ \hline
Upper bound & 85 & 84 & 83 & 82 & 81 & 80 & 79 & 78 & 77 & 76 \\ \hline
Total number  & 1709 & 858 & 387 & 159 & 73 & 68 & 25 & 7 & 2 & 3 \\ \hline
\end{tabular}
\caption{Part \uppercase\expandafter{\romannumeral2}}\label{tab:333-2}
\vspace{-1.5em}
\end{table}

We also calculated ranks of Jacobian matrices for points in  $V(4,4,4;49)$, that are provided in \cite{Faw22+}.  Fawzi et al. provided 14236 points. By computation, we find that the ranks of $J_p(B(4,4,4;49))$ are range from 2144 to 2201. So the upper bounds of local dimensions for these points are range from 151 to 208. Moreover, we find that the rank of $J_p(B(4,4,4;49))$ equal to 2155 for 13530 points, which takes the percentage of 95$\%$.


\subsection{Automorphism groups of $\langle m,n,p\rangle$ and the lower bound}\label{subse:auto}
Let $S(V_1,V_2,V_3)$ denote the group of all \emph{decomposable automorphisms} of $\tilde{V}=V_1\otimes V_2\otimes V_3$. Let $S^0(V_1,V_2,V_3)$ denote the subgroup of $S(V_1,V_2,V_3)$ consisting of all decomposable automorphisms that preserve each factor $V_i$. The set of all decomposable automorphisms of $\tilde{V}$ that preserve $t$ is called \emph{the (full) isotropy group} of $t$, which is denoted by $\Gamma(t)$:
$$\Gamma(t)=\{g\in S(V_1,V_2,V_3)|g(t)=t\}.$$
The \emph{small} isotropy group of $t$ is defined by
$$\Gamma^0(t)=\Gamma(t)\cap S^0(V_1,V_2,V_3).$$
From Corollary \uppercase\expandafter{\romannumeral5}.5 of \cite{deGr87}, we have $\Gamma^0(t)\unlhd\Gamma(t)$ and $\Gamma(t)/\Gamma^0(t)$ is a subgroup of $S_3$.

Let $\mathcal{A}=\{t_1,...,t_r\}$ be an algorithm computing $t$. Then
$$Aut(\mathcal{A})=\{g \in \Gamma(t)| g(\mathcal{A})=\mathcal{A}\}$$
is called \emph{the automorphism group} of $\mathcal{A}$, which is a subgroup of $\Gamma(t)$. The \emph{small} automorphism group of $\mathcal{A}$ is defined by
$$Aut(\mathcal{A})_0=Aut(\mathcal{A}) \cap \Gamma^0(t).$$
Since $\Gamma^0(t)\unlhd \Gamma(t)$ and $Aut(\mathcal{A})\leq\Gamma(t)$, by the second isomorphism theorem of groups (see e.g. Theorem 2.26 of \cite{Rot95}) we have $Aut(\mathcal{A})_0\unlhd Aut(\mathcal{A})$ and
$$Aut(\mathcal{A})/Aut(\mathcal{A})_0=Aut(\mathcal{A})/(Aut(\mathcal{A}) \cap \Gamma^0(t))\cong (Aut(\mathcal{A})\Gamma^0(t))/\Gamma^0(t)\leq \Gamma(t)/\Gamma^0(t)\leq S_3.$$

\begin{proposition}\label{prop-isom}
Let $\mathcal{A}=\{t_1,...,t_r\}$ and $\mathcal{A}'=\{t_1',...,t_r'\}$ be two algorithms computing $t$. If $\mathcal{A}'$ is obtained from
$\mathcal{A}$ by the action of isotropy group $\Gamma(t)$, then
\begin{equation*}
Aut(\mathcal{A})\cong Aut(\mathcal{A}')\quad\text{and}\quad Aut(\mathcal{A})_0\cong Aut(\mathcal{A}')_0.
\end{equation*}
\end{proposition}
\begin{proof}
By assumption, there exists a $g\in \Gamma(t)$, such that
$$t_i'=gt_i,$$
for $i=1,2,...,r$. Then we define
$\phi_g: Aut(\mathcal{A})\to Aut(\mathcal{A}')$ by
\begin{equation*}
\phi_g: a\mapsto gag^{-1}.
\end{equation*}
It is not hard to check that $\phi_g$ is a group isomorphism, whose inverse is given by
\begin{equation*}
\phi_{g}^{-1}=: b\mapsto g^{-1}bg.
\end{equation*}
Moreover, when restricting to $Aut(\mathcal{A})_0$, $\phi_g$ is also a group isomorphism between $Aut(\mathcal{A})_0$ and $Aut(\mathcal{A})_0'$.
\end{proof}


In the following,  let $t=\langle m,n,p \rangle$. We will give an upper bound of $Aut(\mathcal{A})$ (see Corollary \ref{cor:autmnpq}). Recall that
$$\langle m,n,p \rangle=\sum_{i=1}^m\sum_{j=1}^n\sum_{k=1}^p e_{ij}\otimes e_{jk}\otimes e_{ki}\in \mathbb{C}^{m\times n}\otimes \mathbb{C}^{n\times p}\otimes \mathbb{C}^{p\times m}.$$
\begin{definition}\label{def:natalg}
We call $\mathcal{A}(m,n,p)=\{e_{ij}\otimes e_{jk}\otimes e_{ki}|1\leq i\leq m, 1\leq j\leq n, 1\leq k\leq p\}$ the \emph{natural algorithm} computing $\langle m,n,p\rangle$.
\end{definition}
Suppose that $\langle m,n,p \rangle=t_1+t_2+\cdots+t_s$, where $t_i=u_i\otimes v_i\otimes w_i$ and  $u_i\in\mathbb{C}^{m\times n}$, $v_i\in\mathbb{C}^{n\times p}$, $w_i\in\mathbb{C}^{p\times m}$. Let $V=\mathbb{C}^{m\times n}\otimes \mathbb{C}^{n\times p}\otimes \mathbb{C}^{p\times m}$.
For $a \in GL_m(\mathbb{C})$, $b \in GL_n(\mathbb{C})$ and $c \in GL_p(\mathbb{C})$, define transformation $T(a, b, c):V\to V $ by the
formula
\begin{equation}\label{eq-DT}
T(a,b,c)(x\otimes y\otimes z)=axb^{-1}\otimes byc^{-1}\otimes cza^{-1}.
\end{equation}
It was shown in Proposition 4.8 of \cite{Bur15}  that $\Gamma^0(\langle m,n,p \rangle)$ consists of $T(a,b,c)$.

%



Recall that  a \emph{generalized permutation matrix} is a matrix of the form $G=PD$, in which $P, D \in  \mathbb{C}^{n\times n}$, $P$ is a permutation matrix, and $D$ is a nonsingular diagonal matrix \cite{Horn-13}. So $P=(\delta_{i,\pi(j)})$ for some permutation $\pi\in S_n$, where $\delta$ is the Kronecker symbol. Denote the set of generalised permutation matrices of order $n$ by $\mathcal{P}_n$. Recall that $\mathcal{D}_n$ is the set of nonsigular diagonal matrices over $\mathbb{C}$, which is a subgroup of $GL_n(\mathbb{C})$. By the definition of $\mathcal{P}_n$ we have the following proposition.
\begin{proposition}\label{prop:gpiso}
$\mathcal{P}_n$ is isomorphic to the group $\mathcal{D}_n\rtimes S_n$.
\end{proposition}

In the following, if $V$ is a linear space and $\{v_1,v_2,...,v_k\}\subseteq V$ be a subset (may be a multiset) of  vectors, the linear space spanned by $\{v_1,v_2,...,v_k\}$ is denoted by $\langle v_1,v_2,...,v_k\rangle$.
\begin{lemma}\label{le:perm}
Let $L=\{\langle v_1 \rangle,\langle v_2\rangle,...,\langle v_m\rangle\}\subseteq\mathbb{C}^n$ be a subset of one dimensional spaces (lines) of $\mathbb{C}^n$. Let $\{e_i|i=1,2,...,n\}$ be the standard basis of $\mathbb{C}^n$. Suppose that $\langle v_1,v_2,...,v_m\rangle=\mathbb{C}^n$. Suppose that $A\in GL_n(\mathbb{C})$ preserves $L$. Then there exists a $B\in GL_n(\mathbb{C})$ and a $P_A\in\mathcal{P}_n$ such that $A=BP_A B^{-1}$, where $B$ is independent with the choice of $A$.
If $\{e_i|i=1,2,...,n\}\subseteq \{v_1,v_2,...,v_m\}$, then we can choose $B$ as the identity of $GL_n(\mathbb{C})$, so $A=P_A$  is a generalised permutation matrix.
\end{lemma}
\begin{proof}
By assumption we have $n\leq m$. Since the linear span of $\{ v_1,v_2,...,v_m\}$ is $\mathbb{C}^n$, it has a maximal linearly independent subset $M$ which consists of $n$ vectors. Without loss of generality, suppose that $M=\{v_1,v_2,...,v_n\}$.

Suppose that  $A\in GL_n(\mathbb{C})$  preserves $L$. In particular, $A$  preserves the lines $\{\langle v_1\rangle,\langle v_2\rangle,...,\langle v_n \rangle\}$. Let $\tilde{M}=(v_1,v_2,...,v_n)\in \mathbb{C}^{n\times n}$ be the matrix consists of $v_i$ ($i=1,2,...,n$). It implies that
$$\tilde{A}\tilde{M}=(\lambda_1 v_{1},\lambda_2 v_{2},..., \lambda_nv_{n})P =\tilde{M}DP=\tilde{M}PD',$$
where  $P$ is a permutation matrix, $D=diag(\lambda_{1}, \lambda_{2},...,\lambda_{n})$ and $D'=diag(\lambda_{i_1}, \lambda_{i_2},...,\lambda_{i_n})$ ($\lambda_{i_j}$ are obtained by some permutation of $\lambda_i$). Since $v_1,v_2,...,v_n$ are linearly independent, $\tilde{M}$ is invertible. Then we have
\begin{equation}\label{eq:APD}
A=\tilde{M}PD'\tilde{M}^{-1}
\end{equation}
In (\ref{eq:APD}), setting $\tilde{M}=B$ and $PD'=P_A$, we have
$A=BP_AB^{-1}$. Moreover, we can see that $B$ is independent with the choice of $A$. In particular, if $\{e_i|i=1,2,...,n\}\subseteq \{v_1,v_2,...,v_m\}$, we can set $B=(e_1,e_2,...,e_n)$ which is the identity of $GL_n(\mathbb{C})$.
\end{proof}
\begin{remark}\label{rem:conjsg}
Let $L=\{\langle v_1 \rangle,\langle v_2\rangle,...,\langle v_m\rangle\}\subseteq\mathbb{C}^n$ be a subset of one dimensional spaces (lines) with $\langle v_1,v_2,...,v_m\rangle=\mathbb{C}^n$. Let
$$P_L=\{A\in GL_n(\mathbb{C})|~A~\text{preserves}~L\}.$$
So by Lemma \ref{le:perm}, $P_L$ is the conjugate of a subgroup of $\mathcal{P}_n$. Since the conjugate of a subgroup is isomorphic to itself, in the following, we don't distinguish them. So we also write $P_L\leq \mathcal{P}_n$.
\end{remark}

Recall that on $\mathbb{C}^{m\times n}$ (similarly for  $\mathbb{C}^{n\times p}$ and $\mathbb{C}^{p\times m}$) there exists a natural inner product which is defined by $$(a,b)=tr(ab^*),$$
where $tr$ and '*' denote the trace and  conjugate transpose, respectively. With this inner product,  $\mathbb{C}^{m\times n}$  becomes a Hilbert space. Moreover, there is an induced inner product on $\mathbb{C}^{m\times n}\otimes\mathbb{C}^{n\times p}\otimes\mathbb{C}^{p\times n}$, which makes it into a Hilbert space too. Now we have the following lemma
\begin{lemma}\label{le:spuvw}
Suppose that $\langle m,n,p \rangle=t_1+t_2+\cdots+t_r$. Let $t_i=u_i\otimes v_i\otimes w_i$ ($i=1,2,...r$). Then $\langle u_1,u_2,...,u_r\rangle=\mathbb{C}^{m\times n}$, $\langle v_1,v_2,...,v_r\rangle=\mathbb{C}^{n\times p}$ and $\langle w_1,w_2,...,w_r\rangle=\mathbb{C}^{p\times m}$.
\end{lemma}
\begin{proof}
It suffices to show $\langle u_1,u_2,...,u_r\rangle=\mathbb{C}^{m\times n}$, the other two statements are similar.

If $\langle u_1,u_2,...,u_r\rangle\subsetneq\mathbb{C}^{m\times n}$ as a proper linear subspace, there exists a matrix $0\neq a\in \mathbb{C}^{m\times n}$ such that $$(a,u_i)=tr(au_i^*)=0,$$
for $i=1,2,...,r$. Let $a=(a_{ij})$ and suppose that $a_{i_0,j_0}\neq0$ as an entry of $a$. Let $a\otimes e_{j_0,k_0}\otimes e_{k_0, i_0}\in \mathbb{C}^{m\times n}\otimes\mathbb{C}^{n\times p}\otimes\mathbb{C}^{p\times n}$, where $k_0\in \{1,2,...,p\}$ is a fixed number. Then by assumption we have the following inner product equation
\begin{equation}\label{eq:inner}
( a\otimes e_{j_0,k_0}\otimes e_{k_0,i_0},\langle m,n,p\rangle)=
\left( a\otimes e_{j_0,k_0}\otimes e_{k_0, i_0},\sum_{i=1}^r t_i\right).
\end{equation}

Since $\langle m,n,p\rangle=\sum_{i=1}^m\sum_{j=1}^n\sum_{k=1}^p e_{ij}\otimes e_{jk}\otimes e_{ki}$, with the induced inner product on $\mathbb{C}^{m\times n}\otimes\mathbb{C}^{n\times p}\otimes\mathbb{C}^{p\times n}$, the left hand side of (\ref{eq:inner}) is equal to $a_{i_0,j_0}\neq0$. However,  it is not hard to see that the right hand side of (\ref{eq:inner}) is 0, which is a contradiction.
\end{proof}

For $a=(a_{ij})\in \mathbb{C}^{m\times m}$ and $b\in \mathbb{C}^{n\times n}$, we define the tensor product $a\otimes b\in \mathbb{C}^{mn\times mn}$ by
\begin{equation}\label{eq-atb}
a\otimes b=(a_{ij}b).
\end{equation}
Then by the definition of generalised permutation matrix, it is not hard to obtain the following lemma.
\begin{lemma}\label{le:abperm}
If $a\in GL_m(\mathbb{C})$, $b\in GL_n(\mathbb{C})$ such that $a\otimes b\in \mathcal{P}_{mn}$, then $ a\in \mathcal{P}_{m}$ and $b\in \mathcal{P}_{n}$. That is, both $a$ and $b$ are generalised permutation matrices.
\end{lemma}

For $x\in \mathbb{C}^{m\times n}$, define the \emph{rowise vectorization map} $\mathcal{R}: \mathbb{C}^{m\times n}\to \mathbb{C}^{mn\times 1}$ by
concatenating the rows of the matrix in $\mathbb{C}^{m\times n}$ as a column vector of $\mathbb{C}^{mn\times 1}$. So $\mathcal{R}$ is a linear isomorphism. For example, let $a=\left(
                                    \begin{array}{cc}
                                      a_{11} & a_{12} \\
                                      a_{21} & a_{22} \\
                                    \end{array}
                                  \right)\in \mathbb{C}^{2\times2}$, then
$\mathcal{R}(a)=(a_{11}, a_{12}, a_{21}, a_{22})^t\in \mathbb{C}^{4\times1}$, where $t$ is the transpose. For any linear map $\phi$ between $\mathbb{C}^{m\times n}$, $\mathcal{R}$ induces a linear map $\tilde{\phi}$ between $\mathbb{C}^{mn\times 1}$ such that
 $\mathcal{R} \circ \phi=\tilde{\phi}\circ\mathcal{R}$.
In particular, if $\phi(x)=uxv$ where $x\in \mathbb{C}^{m\times n}$, $u\in \mathbb{C}^{m\times m}$ and $v\in \mathbb{C}^{n\times n}$, then $\tilde{\phi}: \mathbb{C}^{mn\times 1}\to \mathbb{C}^{mn\times 1}$ is defined by
\begin{equation}\label{eq:indumap}
\tilde{\phi}:y\mapsto (u\otimes v^t) \left(y\right),
\end{equation}
for each $y=\mathcal{R}(x)\in \mathbb{C}^{mn\times 1}$.

Suppose that $\langle m,n,p \rangle=t_1+t_2+\cdots+t_r$ where $t_i=u_i\otimes v_i\otimes w_i$. Then from Lemma \ref{le:spuvw} we have $\langle u_1,u_2,...u_r\rangle=\mathbb{C}^{m\times n}$. So we can choose a maximal linear independent subset $\{u_{i_1},u_{i_2},...u_{i_{mn}}\}\subseteq\{u_1,u_2,...u_r\}$ such that the following matrix
\begin{equation}\label{eq-umn}
U_{mn}:=\left(\mathcal{R}(u_{i_1}),\mathcal{R}(u_{i_2}),\cdots \mathcal{R}(u_{i_{mn}})\right)\in \mathbb{C}^{mn\times mn}
\end{equation}
is invertible, that is, $U_{mn}\in GL_{mn}(\mathbb{C})$. Similarly, we can define $V_{np}\in GL_{np}(\mathbb{C})$ (resp. $W_{pm}\in GL_{pm}(\mathbb{C})$) for some maximal linear independent subset of $\{ v_1,v_2,...v_r\}$ (resp. $\{ w_1,w_2,...w_r\}$). The set $\{u_1,u_2,...u_r\}$ is said to have \emph{D-property} (Decomposable property) if for some maximal linear independent subset $\{u_{i_1},u_{i_2},...u_{i_{mn}}\}\subseteq\{u_1,u_2,...u_r\}$ we have $U_{mn}=U_m\otimes U_n$ for some  $U_m\in GL_m(\mathbb{C})$ and $U_n\in GL_n(\mathbb{C})$. Similarly, the set $\{v_1,v_2,...v_r\}$ (resp. $\{w_1,w_2,...w_r\}$) has \emph{D-property} if  $V_{np}=V_n\otimes V_p$ (resp. $W_{pm}=W_p\otimes W_m$) where  $V_n\in GL_n(\mathbb{C})$ and $V_p\in GL_p(\mathbb{C})$ (resp. $W_p\in GL_p(\mathbb{C})$ and $W_m\in GL_m(\mathbb{C})$).
\begin{remark}\label{rem-dp0}
Note that $GL_{mn}(\mathbb{C})$ can be identified as an open dense subset of  $\mathbb{C}^{mn\times mn}\cong \mathbb{C}^{m^2\times n^2}$.
For $U_{mn}\in GL_{mn}(\mathbb{C})$, if $U_{mn}=U_m\otimes U_n$ for some  $U_m\in GL_m(\mathbb{C})$ and $U_n\in GL_n(\mathbb{C})$, then $U_{mn}$ belongs to the subset of rank one matrices of $\mathbb{C}^{m^2\times n^2}$, which is sparse. So without considering the structure of the tensor $\langle m,n,p \rangle$,  $\{u_1,u_2,...u_r\}$ with $D$-property may be sparse, similarly for $\{v_1,v_2,...v_r\}$ and $\{w_1,w_2,...w_r\}$.
\end{remark}

\begin{proposition}\label{prop-euk}
Let $\mathcal{A}=\{u_i\otimes v_i\otimes w_i|i=1,2,...,r\}$ be an algorithm computing $\langle m,n,p \rangle$. Then $\{u_i|i=1,2...r\}$ has $D$-property if and only if there exist $a\in GL_m(\mathbb{C})$ and $b\in GL_n(\mathbb{C})$ such that
\begin{equation*}
\{e_{ij}|i=1,2,...,m;~j=1,2,...,n\}\subseteq \{a u_i b^{-1}|i=1,2...r\}.
\end{equation*}
Similar results hold for $\{v_i|i=1,2...r\}$ and $\{w_i|i=1,2...r\}$.
\end{proposition}
\begin{proof}
$\Leftarrow$): Without loss of generality, assume that
\begin{equation*}
\{e_{ij}|i=1,2,...,m;~j=1,2,...,n\}=\{a u_k b^{-1}|k=1,2...mn\}.
\end{equation*}
Then we have
\begin{equation*}
\{ u_k|k=1,2...mn\}=\{a^{-1}e_{ij}b|i=1,2,...,m;~j=1,2,...,n\},
\end{equation*}
and therefore
\begin{equation*}
\{\mathcal{R}(u_k)|k=1,2...mn\}=\{\mathcal{R}(a^{-1}e_{ij}b)
|i=1,2,...,m;~j=1,2,...,n\}.
\end{equation*}
Note that
$$\mathcal{R}(a^{-1}e_{ij}b)=(a^{-1}\otimes b^t) \mathcal{R}(e_{ij})
=(a^{-1}\otimes b^t) e_k,$$
where $k=(i-1)n+j$ and $e_k$ is the $k$th standard basis vector of $\mathbb{C}^{mn}$. So we have
$$\left(\mathcal{R}(u_{k_1}),\mathcal{R}(u_{k_2}),...
\mathcal{R}(u_{k_{mn}})\right)=(a^{-1}\otimes b^t)(e_1,e_2,...,e_{mn})=
a^{-1}\otimes b^t,$$
for some permutation $(k_1,k_2,...,k_{mn})$ of $(1,2,...,mn)$. Thus, $\{u_i|i=1,2...r\}$ has $D$-property.

$\Rightarrow$):
Suppose that $$U_{mn}=\left(\mathcal{R}(u_{1}),\mathcal{R}(u_{2}),\cdots \mathcal{R}(u_{mn})\right)=U_m \otimes U_n,$$
where $U_m\in GL_m(\mathbb{C})$ and $U_n\in GL_m(\mathbb{C})$.


Let $X=U_m^{-1}$ and $Y=U_n^{-1}$. Then we have
$$X\otimes Y (U_{mn})=I_{mn}=(e_1,e_2,...,e_{mn}),$$
where $I_{mn}$ is the $mn$-by-$mn$ identity matrix and $e_k$ ($k=1,2,...,mn$) is the $k$th standard basis vector.
This implies
\begin{equation*}
X\otimes Y \left(\mathcal{R}(u_{k})\right)=e_k.
\end{equation*}

Under $\mathcal{R}$, there is a bijection between $\{e_k|k=1,2,...,mn\}$ and $\{e_{ij}|i=1,2,...,m;~j=1,2,...,n\}$. So if $e_k=\mathcal{R}(e_{ij})$, from (\ref{eq:indumap}), we know that
\begin{equation*}
X u_{k} Y^t =e_{ij}.
\end{equation*}
So the proof is completed by setting $a=X=U_m^{-1}$ and
$b=(Y^t)^{-1}=U_n^t$.
\end{proof}

If an algorithm computing $\langle m,n,p \rangle$ is transformed by the isotropy group action defined in (\ref{eq-DT}), then by Proposition \ref{prop-isom} and \ref{prop-euk} we have the following corollary.
\begin{corollary}
Let $\mathcal{A}=\{u_i\otimes v_i\otimes w_i|i=1,2,...,r\}$ be an algorithm computing $\langle m,n,p \rangle$.
Let $\mathcal{A}'=\{au_ib^{-1}\otimes bv_ic^{-1}\otimes cw_ia^{-1}|i=1,2,...,r\}$, where $a\in GL_m(\mathbb{C})$,  $b\in GL_n(\mathbb{C})$ and $c\in GL_p(\mathbb{C})$. Then
$\mathcal{A}'$ is also an algorithm computing $\langle m,n,p \rangle$ and $$Aut(\mathcal{A}')=Aut(\mathcal{A}).$$
If $\{u_i|i=1,2,...,r\}$ has $D$-property, then we can choose $a$, $b$ such that $$\{e_{i,j}|i=1,2,...,m;~j=1,2,...,n\}\subseteq \{au_ib^{-1}|i=1,2,...,r\}.$$ Similar results hold for $\{v_i|i=1,2,...,r\}$ and $\{w_i|i=1,2,...,r\}$.
\end{corollary}

\begin{remark}\label{rem-dp}
It is interesting to see that for the Laderman algorithm studied in \cite[Sect.5]{Bur15}, $\{e_{i,j}|i=1,2,3;~j=1,2,3\}$ appears in
$\{u_i\}$, $\{v_i\}$ and $\{w_i\}$. With the help of computer, without taking the transformation (\ref{eq-DT}), we check if other algorithms provided in the database \cite{Heule19} contain all $\{e_{i,j}|i=1,2,3;~j=1,2,3\}$. However, no such algorithms are found.

Similarly, we also check the database \cite{Faw22+} for $\langle4,4,4\rangle$. The algorithms which contain all $\{e_{i,j}|i=1,2,3,4;~j=1,2,3,4\}$ are not found. So it is interesting to find more algorithms which contain all $\{e_{i,j}|i=1,2,3;~j=1,2,3\}$ (resp. $\{e_{i,j}|i=1,2,3,4;~j=1,2,3,4\}$) for $\langle3,3,3\rangle$
(resp. for  $\langle4,4,4\rangle$), etc.
\end{remark}

\begin{definition}\label{def:d-p}
Given a tensor decomposition $\langle m,n,p \rangle=u_1\otimes v_1\otimes w_1+u_2\otimes v_2\otimes w_2+\cdots u_r\otimes v_r\otimes w_r$. Let $\mathcal{A}=\{u_i\otimes v_i\otimes w_i|i=1,2,...,r\}$ be the corresponding algorithm computing $\langle m,n,p \rangle$. We call $\mathcal{A}$ has \emph{D-property} if two of the three sets $\{u_1,u_2,...u_r\}$, $\{v_1,v_2,...v_r\}$ and $\{w_1,w_2,...w_r\}$ have $D$-property.
\end{definition}

The  $D$-property defined above seems so special. However, we have some examples.
\begin{example}\label{exp:Dp}
It is not hard to see that the natural algorithm has $D$-property.
Because after a suitable change of the order of the maximal independent subset, we can make $U_{mn}=I_{mn}=I_m\otimes I_n$,  $V_{np}=I_{np}=I_n\otimes I_p$ and $W_{pm}=I_{pm}=I_p\otimes I_m$, where $I_{mn}$ is the identity of the group $GL_{mn}(\mathbb{C})$ and similarly, for others. The Laderman algorithm studied in Section 5 of \cite{Bur15} also has this property.
\end{example}

\begin{theorem}\label{th:auta0}
Let $\mathcal{A}=\{u_i\otimes v_i\otimes w_i|i=1,2,...,r\}$ be an algorithm computing $\langle m,n,p \rangle$. If $\mathcal{A}$ has D-property, then (with the convention in Remark \ref{rem:conjsg})
$Aut(\mathcal{A})_0\leq\mathcal{P}_m\times \mathcal{P}_n\times \mathcal{P}_p$.
\end{theorem}
\begin{proof}
By assumption, $Aut(\mathcal{A})_0=\{T(a,b,c)\in \Gamma^0(\langle m,n,p \rangle)| T(a,b,c)(\mathcal{A})=\mathcal{A}\}$.
So if $T(a,b,c)\in Aut(\mathcal{A})_0$, then $\mathcal{A}=\{t_i|i=1,2,...,r\}=\{u_i\otimes v_i\otimes w_i|i=1,2,...,r\}=\{au_ib^{-1}\otimes bv_ic^{-1}\otimes cw_ia^{-1}|i=1,2,...,r\}$. Without loss of generality, suppose that $\{u_i|i=1,2,...,r\}$ and $\{v_i|i=1,2,...,r\}$ have $D$-property.

By Lemma 5.9 of \cite{Bur15}, we have $\varphi(x)=axb^{-1}$ (where $x\in\mathbb{C}^{m\times n}$) is the linear map that preserves the set of lines
$$U=\{\langle u_i \rangle|i=1,2,...,r\},$$
where $\langle u_i \rangle$ are the one dimensional spaces (lines) spanned by $u_i$ in $\mathbb{C}^{m\times n}$.
Let $\mathcal{R}(x)\in \mathbb{C}^{mn\times 1}$ be the rowise vectorization of $x\in \mathbb{C}^{m\times n}$. Let $\tilde{\varphi}(y)=[a\otimes (b^{-1})^t] (y)$ for $y\in \mathbb{C}^{mn\times 1}$ be the induced map defined in (\ref{eq:indumap}).
Then we have $\varphi(x)=axb^{-1}$ if and only if $\tilde{\varphi}(\mathcal{R}(x))=[a\otimes (b^{-1})^t] (\mathcal{R}(x))$. Let
$$\tilde{U}=\{\langle \mathcal{R}(u_i) \rangle\subseteq \mathbb{C}^{mn\times 1} |i=1,2,...,r\}.$$
From Lemma \ref{le:spuvw}, we have  $\langle u_1,u_2,...,u_r\rangle=\mathbb{C}^{m\times n}$.  So
$\langle \mathcal{R}(u_1),\mathcal{R}(u_2),...,\mathcal{R}(u_r) \rangle=\mathbb{C}^{mn\times 1}$.
Then $\varphi$ preserves $U$ if and only if $\tilde{\varphi}$ preserves $\tilde{U}$. So by Lemma \ref{le:perm}, we have
\begin{equation}\label{eq-abpmn}
a\otimes (b^{-1})^t=U_{mn} P_{ab} U_{mn}^{-1},
\end{equation}
for some $U_{mn}\in GL_{mn}(\mathbb{C})$ and $P_{ab}\in \mathcal{P}_{mn}$. Since $\{u_1,u_2,...u_r\}$ has $D$-property, we can set $U_{mn}=U_m\otimes U_n$ for some $U_m\in GL_{m}(\mathbb{C})$ and
$U_n\in GL_{n}(\mathbb{C})$. So we have
$$P_{ab}=\left(U_m^{-1}a U_m\right)\otimes \left(U_n^{-1} (b^{-1})^t U_n\right).$$
Thus, by Lemma \ref{le:abperm}, we have
$$U_m^{-1}a U_m=P_a\quad\text{and}\quad U_n^{-1} (b^{-1})^t U_n=P_b,$$
for some $P_a\in \mathcal{P}_m$ and $P_b\in \mathcal{P}_n$. So
$a=U_m P_a U_m^{-1}\in U_m\mathcal{P}_m U_m^{-1}$ and $b=(U_n^t)^{-1} (P_b^t)^{-1} U_n^t\in (U_n^t)^{-1}\mathcal{P}_n U_n^t$.

On the other hand, $\phi(y)=byc^{-1}$
(where $y\in\mathbb{C}^{n\times p}$) is the linear map that preserves the set of lines
$$V=\{\langle v_i \rangle|i=1,2,...,r\}.$$

Since  $\langle v_1,v_2,...,v_r\rangle=\mathbb{C}^{n\times p}$, as (\ref{eq-abpmn}), we have
\begin{equation*}
b\otimes (c^{-1})^t=V_{np} P_{bc} V_{np}^{-1},
\end{equation*}
for some $V_{np}\in GL_{np}(\mathbb{C})$ and $P_{bc}\in \mathcal{P}_{np}$. Since $\{v_1,v_2,...v_r\}$ also has $D$-property, we can set $V_{np}=V_n\otimes V_p$ for some $V_n\in GL_{n}(\mathbb{C})$ and
$V_p\in GL_{p}(\mathbb{C})$. Thus,
$c=(V_p^t)^{-1} (P_c^t)^{-1} V_p^t\in (V_p^t)^{-1}\mathcal{P}_p V_p^t$, for some $P_c\in\mathcal{P}_p$.


So from discussions above, we have
$Aut(\mathcal{A})_0\leq\mathcal{P}_m\times \mathcal{P}_n\times \mathcal{P}_p$.
\end{proof}



Let $Q(m,n,p)$ be a subgroup of $\Gamma(\langle m,n,p \rangle)$, isomorphic to $S_3$, $\mathbb{Z}_2$, or $1$, when $|\{m, n, p\}| $= 1, 2, or 3, respectively. Just as in the proof of Theorem 4.12 and Lemma 5.8 of \cite{Bur15}, $Aut(\mathcal{A})$ is the semidirect product of $Aut(\mathcal{A})_0$ with a subgroup of $Q(m,n,p)$. By Theorem \ref{th:auta0}, we have the  following corollary.
\begin{corollary}\label{cor:autmnpq}
Suppose that $\mathcal{A}$ is an algorithm computing $\langle m,n,p \rangle$. If $\mathcal{A}$ has $D$-property, then
$Aut(\mathcal{A})\leq(\mathcal{P}_m\times \mathcal{P}_n\times \mathcal{P}_p)\rtimes Q(m,n,p)$.
\end{corollary}


For the natural algorithm, we have the following corollary.
\begin{corollary}\label{cor:autnat}
$Aut(\mathcal{A}(m,n,p))=(\mathcal{P}_m\times \mathcal{P}_n\times \mathcal{P}_p)\rtimes Q(m,n,p)$.
\end{corollary}
\begin{proof}
From Example \ref{exp:Dp}, we know that $\mathcal{A}(m,n,p)$ has $D$-property. So from Corollary \ref{cor:autmnpq}, in the following we only need to check that $(\mathcal{P}_m\times \mathcal{P}_n\times \mathcal{P}_p)\rtimes Q(m,n,p)$ is contained in $Aut(\mathcal{A}(m,n,p))$.

Let $a=P_1D_1\in \mathcal{P}_m$ where $P_1=(\delta_{i,\pi(j)})$ is a permutation matrix corresponding to the permutation $\pi\in S_m$  and $D_1=diag(a_1,a_2,...,a_m)$ is a nonsingular diagonal matrix with $a_i\in \mathbb{C}\setminus \{0\}$. Similarly, let $b=P_2D_2\in \mathcal{P}_n$ where $P_2=(\delta_{i,\sigma(j)})$, $D_2=diag(b_1,b_2,...,b_n)$ and $\sigma\in S_n$. Let
$c=P_3D_3\in \mathcal{P}_p$ where $P_2=(\delta_{i,\tau(j)})$, $D_3=diag(c_1,c_2,...,c_p)$  and $\tau\in S_p$.

Then we can see that
\begin{align*}
T(a,b,c)(e_{ij}\otimes e_{jk}\otimes e_{ki})&=a_i \left( e_{\pi^{-1}(i),\sigma^{-1}(j)}\right) b_j^{-1}\otimes b_j \left( e_{\sigma^{-1}(j),\tau^{-1}(k)}\right)c_k^{-1}\otimes c_k \left( e_{\tau^{-1}(k),\pi^{-1}(i)}\right)a_i^{-1}\\
   &=e_{\pi^{-1}(i),\sigma^{-1}(j)}\otimes e_{\sigma^{-1}(j),\tau^{-1}(k)}\otimes e_{\tau^{-1}(k),\pi^{-1}(i)}.
\end{align*}
Hence,
\begin{align*}
T(a,b,c)\mathcal{A}(m,n,p)&=\{e_{\pi^{-1}(i),\sigma^{-1}(j)}\otimes e_{\sigma^{-1}(j),\tau^{-1}(k)}\otimes e_{\tau^{-1}(k),\pi^{-1}(i)}|1\leq i\leq m, 1\leq j\leq n, 1\leq k\leq p\}\\
   &=\{e_{ij}\otimes e_{jk}\otimes e_{ki}|1\leq i\leq m, 1\leq j\leq n, 1\leq k\leq p\}\\
   &=\mathcal{A}(m,n,p).
\end{align*}
Thus, by Theorem \ref{th:auta0} we have $Aut(\mathcal{A}(m,n,p))_0=\mathcal{P}_m\times \mathcal{P}_n\times \mathcal{P}_p$.

Since $Aut(\mathcal{A}(m,n,p))$ is also preserved under the action of $Q(m,n,p)$ (see e.g. Theorem 4.12 of \cite{Bur15}), we have
$(\mathcal{P}_m\times \mathcal{P}_n\times \mathcal{P}_p)\rtimes Q(m,n,p)\subseteq Aut(\mathcal{A}(m,n,p))$.
\end{proof}

\begin{remark}
Conversely, it is interesting to show that if $Aut(\mathcal{A})=(\mathcal{P}_m\times \mathcal{P}_n\times \mathcal{P}_p)\rtimes Q(m,n,p)$, then $\mathcal{A}$ should be the natural algorithm $\mathcal{A}(m,n,p)$.
\end{remark}
By the discussions in Section 4.3 of \cite{Bur15}, we have the following proposition.
\begin{proposition}\cite[Sect.4.3]{Bur15}\label{prop:isomnp}
$\Gamma^0(\langle m,n,p \rangle)=PGL_m(\mathbb{C})\times  PGL_n(\mathbb{C})\times PGL_p(\mathbb{C})$ and $\Gamma(\langle m,n,p \rangle)=\Gamma^0(\langle m,n,p \rangle)\rtimes Q(m,n,p)$.
\end{proposition}
For a point  $q\in V(m,n,p;r)$, by (\ref{eq:mnpr}) and (\ref{eq:brent}) there exists a tensor decomposition of $\langle m,n,p \rangle$ of length $r$:
\begin{equation}\label{eq:tq}
\langle m,n,p \rangle=t^q_1+t^q_2+\cdots+t^q_r,
\end{equation}
where $t^q_i=u_i^q\otimes v_i^q\otimes w_i^q\in \mathbb{C}^{m\times n}\otimes \mathbb{C}^{n\times p}\otimes\mathbb{C}^{p\times m}$.
\begin{corollary}\label{cor:lbbe}
For a point $q\in V(m,n,p;r)$, with notations in (\ref{eq:tq}) let
$\mathcal{A}^q=\{t^q_1,t^q_2,...,t^q_r\}$ be the corresponding
 algorithm of length $r$ computing $\langle m,n,p \rangle$.
If $\mathcal{A}^q$ has D-property, then
$$\dim_q V(m,n,p;r)\geq m^2+n^2+p^2-m-m-p-3.$$
\end{corollary}
\begin{proof}
Since $\mathcal{A}^q$ has D-property, $\mathcal{A}^q\leq (\mathcal{P}_m\times \mathcal{P}_n\times \mathcal{P}_p)\rtimes Q(m,n,p)$ by Corollary \ref{cor:autmnpq}. So by Proposition \ref{prop:dimag} and \ref{prop:gpiso}, we have
$$\dim Aut(\mathcal{A}^q)\leq m+n+p.$$
By the analysis in Subsection \ref{subsec:isogolb},  the isotropy group $\Gamma(\langle m,n,p \rangle)$ induces a group action on
$V(m,n,p;r)$. Let $\Gamma(\langle m,n,p \rangle)\cdot q$ denote the group orbit of $q$, which is induced by the action of $\Gamma(\langle m,n,p \rangle)$ on $V(m,n,p;r)$.
Then by Proposition \ref{prop:dimag}, \ref{prop:lowerb} and \ref{prop:isomnp}, we have
\begin{align*}
\dim_q V(m,n,p;r)&\geq
   \dim\left(\Gamma(\langle m,n,p \rangle)\cdot q\right)\\
   &=\dim\left(\Gamma(\langle m,n,p \rangle)\right)-\dim Aut(\mathcal{A}^q)\\
   &\geq m^2+n^2+p^2-3-(m+n+p).
\end{align*}
\end{proof}

By Remark \ref{rem-dp0} and \ref{rem-dp}, we know the algorithms have $D$-property may be sparse. In the following, we point out another property which also implies the result of Theorem \ref{th:auta0}, Corollary \ref{cor:autmnpq} and \ref{cor:lbbe}. We find that many algorithms of \cite{Faw22+} and \cite{Heule19} have this property.
%

Suppose that $\langle m,n,p \rangle=u_1\otimes v_1\otimes w_1+u_2\otimes v_2\otimes w_2+\cdots u_r\otimes v_r\otimes w_r$. Following the notations of (\ref{eq-umn}), we say that $U_{mn}$ has \emph{weak $D$-property} if it is a block diagonal matrix of the form
\begin{equation}\label{eq-wdpu}
U_{mn}=\left(
  \begin{array}{cc}
    U_{mn}^{(1)} & 0 \\
    0 & U_{mn}^{(2)} \\
  \end{array}
\right),
\end{equation}
where $U_{mn}^{(1)}\in \mathbb{C}^{n\times n}$ and $U_{mn}^{(2)}\in \mathbb{C}^{(mn-n)\times (mn-n)}$.
Since $U_{mn}\in GL_{mn}(\mathbb{C})$, we have $U_{mn}^{(1)}\in GL_{n}(\mathbb{C})$ and $U_{mn}^{(2)}\in GL_{mn-n}(\mathbb{C})$.
We can also write (\ref{eq-wdpu}) as $U_{mn}=U_{mn}^{(1)}\oplus U_{mn}^{(2)}$.

%

Similarly, we say that $V_{np}$ has \emph{weak $D$-property} if
\begin{equation*}
V_{np}=V_{np}^{(1)}\oplus V_{np}^{(2)},~\text{for some}~V_{np}^{(1)}\in \mathbb{C}^{p\times p},~\text{and}~V_{np}^{(2)}\in \mathbb{C}^{(np-p)\times (np-p)}.
\end{equation*}
And, $W_{pm}$ has \emph{weak $D$-property} if
\begin{equation*}
W_{pm}=W_{pm}^{(1)}\oplus W_{pm}^{(2)},~\text{for some}~W_{pm}^{(1)}\in \mathbb{C}^{m\times m},~\text{and}~W_{pm}^{(2)}\in \mathbb{C}^{(pm-m)\times (pm-m)}.
\end{equation*}

\begin{definition}\label{def-wdp}
Given a tensor decomposition $\langle m,n,p \rangle=u_1\otimes v_1\otimes w_1+u_2\otimes v_2\otimes w_2+\cdots u_r\otimes v_r\otimes w_r$. Let $\mathcal{A}=\{u_i\otimes v_i\otimes w_i|i=1,2,...,r\}$ be the corresponding algorithm computing $\langle m,n,p \rangle$. We call $\mathcal{A}$ has \emph{weak D-property} if both $\{u_1,u_2,...u_r\}$, $\{v_1,v_2,...v_r\}$ and $\{w_1,w_2,...w_r\}$ have weak $D$-property.
\end{definition}

\begin{remark}

With the help of computer, in the database \cite{Heule19} we find that
 there are 8664 algorithms which have weak $D$-property among all  17376 algorithms, which take the proportion of nearly 50\%; in the database \cite{Faw22+}, for all 14236 algorithms of length 49 computing $\langle 4,4,4\rangle$, 14204 algorithms have weak $D$-property, which take the proportion of 99\%.
\end{remark}

\begin{proposition}\label{prop-wdaut}
Let $\mathcal{A}=\{u_i\otimes v_i\otimes w_i|i=1,2,...,r\}$ be an algorithm computing $\langle m,n,p \rangle$ and $q\in V(m,n, p;r)$ be the corresponding solution of Brent equations. Suppose that $\mathcal{A}$ has weak $D$-property. Then (with the convention in Remark \ref{rem:conjsg})
$Aut(\mathcal{A})_0\leq\mathcal{P}_m\times \mathcal{P}_n\times \mathcal{P}_p.$
Moreover, we also have $Aut(\mathcal{A})\leq(\mathcal{P}_m\times \mathcal{P}_n\times \mathcal{P}_p)\rtimes Q(m,n,p)$ and
$$\dim_q V(m,n,p;r)\geq m^2+n^2+p^2-m-m-p-3.$$
\end{proposition}

\begin{proof}
By assumption, both $\{u_i|i=1,2,...,r\}$, $\{v_i|i=1,2,...,r\}$ and $\{w_i|i=1,2,...,r\}$ have weak $D$-property.

Since $Aut(\mathcal{A})_0=\{T(a,b,c)\in \Gamma^0(\langle m,n,p \rangle)| T(a,b,c)(\mathcal{A})=\mathcal{A}\}$, so if $T(a,b,c)\in Aut(\mathcal{A})_0$, then $\mathcal{A}=\{t_i|i=1,2,...,r\}=\{u_i\otimes v_i\otimes w_i|i=1,2,...,r\}=\{au_ib^{-1}\otimes bv_ic^{-1}\otimes cw_ia^{-1}|i=1,2,...,r\}$. In the following, we will show that if $\{u_i|i=1,2,...,r\}$ has weak $D$-property then
\begin{equation}\label{eq-bupu}
b=U^{-1}P_bU,
\end{equation}
for some $U\in GL_{n}(\mathbb{C})$ which is independent with the choice of $b$, and $P_b\in \mathcal{P}_n$.

In fact, by (\ref{eq-abpmn}) in the proof of Theorem \ref{th:auta0}, we
have
\begin{equation}\label{eq-abpmn+}
a\otimes (b^{-1})^t=U_{mn} P_{ab} U_{mn}^{-1},
\end{equation}
for some $U_{mn}\in GL_{mn}(\mathbb{C})$ and $P_{ab}\in \mathcal{P}_{mn}$. Since $\{u_i|i=1,2,...,r\}$ has weak $D$-property, we have
\begin{equation*}
U_{mn}=U_{mn}^{(1)}\oplus U_{mn}^{(2)}
\end{equation*}
for some $U_{mn}^{(1)}\in GL_{n}(\mathbb{C})$ and $U_{mn}^{(2)}\in GL_{mn-n}(\mathbb{C})$. Write $a$ as $a=(a_{ij})_{1\leq i,j\leq m}$. By (\ref{eq-atb}) and (\ref{eq-abpmn+}), we have
$$ a_{11}(b^{-1})^t= U_{mn}^{(1)} \tilde{P}_{b} U_{mn}^{(1)^{-1}}, $$
where $\tilde{P}_{b}$ is a submatrix of $P_{ab}$ which is also a generalized permutation matrix. So we have
$$ b=a_{11} (U_{mn}^{(1)t})^{-1} (\tilde{P}_{b}^{t})^{-1} U_{mn}^{(1)t}, $$
where $U_{mn}^{(1)t}$ is the transpose of $U_{mn}^{(1)}$. Since $a_{11}(\tilde{P}_{b}^{t})^{-1}$ is also a generalized permutation matrix, by setting $U=U_{mn}^{(1)t}$ and $P_b=a_{11}(\tilde{P}_{b}^{t})^{-1}$ we  obtain (\ref{eq-bupu}).

By similar argument, we can also show that: if $\{v_i|i=1,2,...,r\}$ has weak $D$-property then $c=V^{-1}P_cV,$
for some $V\in GL_{p}(\mathbb{C})$ and $P_c\in \mathcal{P}_p$;
if $\{w_i|i=1,2,...,r\}$ has weak $D$-property then $a=W^{-1}P_aW,$
for some $W\in GL_{m}(\mathbb{C})$ and $P_a\in \mathcal{P}_m$.

So combining discussions above, we have
$Aut(\mathcal{A})_0\leq\mathcal{P}_m\times \mathcal{P}_n\times \mathcal{P}_p$. Moreover, the proofs of $Aut(\mathcal{A})\leq(\mathcal{P}_m\times \mathcal{P}_n\times \mathcal{P}_p)\rtimes Q(m,n,p)$ and
$\dim_q V(m,n,p;r)\geq m^2+n^2+p^2-m-m-p-3$ are the same as Corollary \ref{cor:autmnpq} and \ref{cor:lbbe}.
\end{proof}

By current results on automorphism groups (see e.g. \cite{Bur15}, \cite{Bur22} and \cite{Sedogl-23}), we believe the $D$-property and weak $D$-property in Corollary \ref{cor:autmnpq} and Proposition \ref{prop-wdaut} are not necessary. So we have the following conjecture for any algorithms of (any) length $r$ computing $\langle m,n,p \rangle$, that is, $r$ can be less or more than $mnp$.
\begin{conjecture}
Suppose that $\mathcal{A}$ is an algorithm of any length $r$ computing $\langle m,n,p \rangle$. Then
$$Aut(\mathcal{A})\leq(\mathcal{P}_m\times \mathcal{P}_n\times \mathcal{P}_p)\rtimes Q(m,n,p).$$
\end{conjecture}

\subsection{On the gap between the lower and upper bounds}

Given a point $q\in V(m,n,p;r)$. Let $\mathcal{A}^q$
be the corresponding algorithm of length $r$ computing $\langle m,n,p\rangle$.
From Theorem \ref{th:updim}, we know
the upper bound of $\dim_q V(m,n,p;r)$ provided by the Jacobian Criterion is
$$u_q(m,n,p;r)=k-rank (J_q(B(m,n, p;r))),$$
where $k=(mn+np+pm)r$.  Suppose that $\mathcal{A}^q$ has $D$ or weak $D$-property. Then from Corollary \ref{cor:lbbe} and Proposition \ref{prop-wdaut},
the lower bound of $\dim_q V(m,n,p;r)$ provided by the isotropy group action  is
$$l_q(m,n,p;r)=m^2+n^2+q^2-m-n-p-3.$$

Besides the (full) isotropy group studied in this paper (see also \cite[Sect.4]{Bur15}), there are two other (trivial) linear algebraic group actions on $\mathcal{A}^q$ which all tensors of order 3 have (see Remark 5.1 of \cite{LZK-23}\footnote{We are grateful to the anonymous referee of \cite{LZK-23} for pointing out this.}). Then  $l_q(m,n,p;r)$ can be improved as
\begin{equation}\label{eq-lqq}
l_q'(m,n,p;r)=m^2+n^2+q^2+2r-m-n-p-3.
\end{equation}
From the discussions in \cite{Bur14,Bur15} and \cite{Sedogl-23}, in many cases, the automorphism group (or stabilizer) of $\mathcal{A}^q$ is finite. Then $l_q'(m,n,p;r)$ can be further improved as
\begin{equation}\label{eq-lqqq}
l_q''(m,n,p;r)=m^2+n^2+q^2+2r-3.
\end{equation}

We let the number
$$g_q(m,n,p;r)=u_q(m,n,p;r)-l_q(m,n,p;r),$$
denote the gap between $u_q(m,n,p;r)$ and $l_q(m,n,p;r)$.
And similarly, denote
$$g_q'(m,n,p;r)=u_q(m,n,p;r)-l_q'(m,n,p;r),$$
and
\begin{equation}\label{eq-gqq}
g_q''(m,n,p;r)=u_q(m,n,p;r)-l_q''(m,n,p;r).
\end{equation}

Since the Jacobian matrix $J_q(B(m,n, p;r)\in \mathbb{C}^{(mnp)^2\times k}$ where $k=(mn+np+pm)r$, when $m=n=p$ the upper bound $u_q(m,n,p;r)$ is of the order of $n^5$, and the lower bound $l_q(m,n,p;r)$ is of the order of $n^2$. By this observation, the gap between them is big. So give estimations and find methods to narrow the gap are natural questions. Recently, in \cite{LZK-23},  using the deflation method in numerical algebraic geometry, we find that $u_q(m,n,p;r)$ can be improved. For solutions in the database \cite{Heule19} and \cite{Faw22+}, their upper bounds can be decreased during the deflation processes. So the gaps are decreased, too.

Another question is the relation between the gap and the tensor rank.
Fixing $m,n,p$, if $ V(m,n,p;r)\neq\emptyset$, as $r$ grows we know that
$$ V(m,n,p;r)\subsetneq V(m,n,p;r+1)\subsetneq\cdots$$
Therefore, $g_q(m,n,p;r)$, $g_q'(m,n,p;r)$ and $g_q''(m,n,p;r)$ will become larger as $r$ increased. Hence, they will become smaller as $r$ decreased. If $q$ is Strassen's algorithm \cite{Stra-69},
by (\ref{eq-gqq}) we have $g_q''(2,2,2;7)=0$ (see also \cite[Sect.5.1]{LZK-23}).
It is well known that the tensor rank of $\langle 2,2,2\rangle$ is 7.
For $q\in V(3,3,3;23)$, from (\ref{eq-lqqq}) we have $l_q''(3,3,3;23)=70$. In
\cite[Sect 5.2]{LZK-23}, by the deflation method, for all $q$ in the database \cite{Heule19}, we find that $70\leq u_q(3,3,3;23)\leq 86$.
So $0\leq g_q''(3,3,3;23)\leq 16$ which is small.

We summarize above discussions as the following questions.
\begin{question} Fix $m$, $n$ and $p$. Let
$R(\langle m,n,p\rangle)$ be the tensor rank of $\langle m,n,p\rangle$.

\begin{enumerate}
  \item For  $r\geq R(\langle m,n,p\rangle)$ and $q\in V(m,m,p;r)$, give  estimations for the gaps $g_q(m,n,p;r)$, $g_q'(m,n,p;r)$ and $g_q''(m,n,p;r)$.

  \item When $r=R(\langle m,n,p\rangle)$, are the gaps above equal to 0?
\end{enumerate}

%
\end{question}


\subsection{Radical ideal and rank of $J_q(B(m,n,p;r))$ as an invariant}

Let $I(2,2,2;7)$ denote the set $$\{f\in \mathbb{C}[x_1,x_2,...,x_{84}]| f(p)=0~\text{when}~p\in V(2,2,2;7)\}.$$ It is the vanishing ideal of $V(2,2,2;7)$ which is radical. On the other hand, let $\langle B(2,2,2;7) \rangle$ denote the ideal generated by the Brent equations $B(2,2,2;7)$.
We have that $\langle B(2,2,2;7) \rangle$ is strictly contained in  $I(2,2,2;7)$. The reason is as follows.

Suppose that $I(2,2,2;7)=\langle f_1, f_2,...,f_k\rangle$ where $f_i\in \mathbb{C}[x_1,x_2,...,x_{84}]$.
It was shown in \cite{deGr78-2} that the isotropy group $\Gamma(\langle2,2,2\rangle)$ acts transitively on $V(2,2,2;7)$ and $\dim V(2,2,2;7)=9$. So by Proposition \ref{prop:ldimgo}, $V(2,2,2;7)$ is a smooth affine variety, which has no singular points. So we have $\dim V(2,2,2;7)=84-rank~J_p(f_1,f_2,...,f_k)$ for $p\in V(2,2,2;7)$. Since the Brent equations $B(2,2,2;7)$ have 64 equations and 84 variables, the rank of $J_p(B(2,2,2;7))$ can not exceed 64. So $84-rank~J_p(B(2,2,2;7))\geq 20>9$. So we have
$\langle B(2,2,2;7) \rangle \subsetneq I(2,2,2;7)$.
In fact,  we calculated that $rank~J_p(B(2,2,2;7))=61$ when $p$ is Strassen's algorithm \cite{Stra-69}.
So the ideal generated by Brent equations may not be radical.

Generally, let $k=(mn+np+pm)r$ and $I(m,n,p;r)$ denote the set $$\{f\in \mathbb{C}[x_1,x_2,...,x_{k}]| f(p)=0~\text{when}~p\in V(m,n,p;r)\},$$ that is, the vanishing ideal of $V(m,n,p;r)$. Let
$\langle B(m,n,p;r) \rangle$ denote the ideal generated by the Brent equations $B(m,n,p;r)$. By the discussions above, $\langle B(m,n,p;r) \rangle$ may not be equal to $I(m,n,p;r)$. So from Corollary \ref{cor:rkorb}, we can see that  $rank~J_q(B(m,n,p;r))$ may not be equal to $rank~J_{q'}(B(m,n,p;r))$ when $q$ and $q'$ lie in the same group orbit under the action of $\Gamma(\langle m,n,p\rangle)$ on $V(m,n,p;r)$. However,  by random experiments in computer we have not found counterexamples. So we have the following question
\begin{question}\label{qu:rkob}
Does $rank~J_q(B(m,n,p;r))=rank~J_{q'}(B(m,n,p;r))$ when $q$ and $q'$ lie in the same group orbit under the action of $\Gamma(\langle m,n,p\rangle)$ on $V(m,n,p;r)$?
\end{question}
If the answer of Question \ref{qu:rkob} is affirmative, we get another criterion which can be used to distinguish group orbits.
\begin{remark}
Since the rank of $J_q(B(m,n,p;r))$ is finite and the number of group orbits may be infinite \cite{Johnson-86}, it has limitation when distinguishing group orbits just by the rank of $J_q(B(m,n,p;r))$.

By now, some efficient algorithms and criterions  are raised to distinguish group orbits, such as \cite{Berger22,Heule21,Kauers} and the supplementary information of \cite{Faw22}. To distinguish the group orbits more efficiently, we should combine all these methods.
\end{remark}

\section{Concluding Remarks}
Through the solutions of Brent equations, we can find new fast matrix multiplication algorithms. In this paper, we discuss the local dimensions of these solutions over the complex field. Our tools are the classical results from  algebraic geometry and linear algebraic groups.

Since the scale of Brent equations is large, it is not so easy to determine the local dimensions explicitly. However, give estimations for the local dimensions are feasible, such as the lower and upper bounds.
To obtain the upper bounds for local dimensions, we use
the Jacobian Criterion. Since there is an isotropy group action on the solution set, using the theories in linear algebraic groups we can obtain the lower bounds. In this case, the key step is to characterize the automorphism group (or the stabilizer) of the group action.
From the discussions in \cite{Bur14,Bur15,Bur22} and \cite{Sedogl-23}, we can see that many automorphism groups are finite. In this paper, we only give upper bound of automorphism groups under some special conditions. So further discussions on the automorphism groups are expected, such as when they are finite.

Finally, we hope our work can provide some new insights for the behavior of the solution set of Brent equations.
\section*{Acknowledgments}
We are very grateful to the editors and referees for their valuable comments and suggestions.

X. Li is supported by National Natural Science Foundation of China (Grant No.11801506) and the National Science Basic Research Plan in Shaanxi Province of China (2021JQ-661). Y. Bao is supported by the NSFC (Grant No.11801117) and the Natural Science Foundation of Guangdong Province. China (Grant No. 2018A030313268). L. Zhang is
supported by the Zhejiang Provincial Natural Science Foundation of China (Grant No.LY21A010015) and Grant 11601484, the National Natural Science Foundation, Peoples Republic of China.
Many thanks to Xiaodong Ding and Yuan Feng. X. Li is grateful to Yaqiong Zhang for her support and patience.

\bibliographystyle{amsplain}

\end{document}